\documentclass{amsart}

\newif\ifdraft\draftfalse
\newif\ifcite\citefalse
\newif\ifblow\blowtrue

\usepackage{amsfonts,amssymb, amsmath}
\usepackage{oldgerm}
\usepackage{amscd}
\usepackage{eucal} 
\usepackage{mathrsfs} 
\usepackage[hyperindex]{hyperref}  
\usepackage[alphabetic]{amsrefs}
\ifdraft\ifcite\usepackage{showkeys}\else\usepackage[notcite,notref]{showkeys}\fi\fi

\usepackage{xypic}
\newdir{ >}{{}*!/-5pt/\dir{>}}

\newtheorem{proposition}[equation]{Proposition}
\newtheorem{theorem}[equation]{Theorem}

\newtheorem{lemma}[equation]{Lemma}

\theoremstyle{remark}

\theoremstyle{definition}

\theoremstyle{remark}

\newtheorem{remark}[equation]{Remark}

\newtheorem{rem}[equation]{Remark}

\numberwithin{equation}{section}

\def\bc{\begin{cases}}
\def\ec{\end{cases}}

\def\a{\alpha}
\def\Bbb{\mathbb}

\def\cp{{\mathcal P}}

\def\cu{{\mathcal U}}

\def\cw{{\mathcal W}}

\def\bc{{\mathbb C}}

\def\er{\eqref}

\def\bc{\mathbb C}

\def\lp2{L_pH_{2p}}
\def\bean{\begin{eqnarray}}
\def\eean{\end{eqnarray}}
\def\bea{\begin{eqnarray*}}
\def\eea{\end{eqnarray*}}
\def\beq{\begin{equation}}
\def\eeq{\end{equation}}
\def\beq*{\begin{equation*}}
\def\eeq*{\end{equation*}}
\def\bal{\begin{align*}}
\def\eal{\end{align*}}
\def\baln{\begin{align}}
\def\ealn{\end{align}}

\def\beg{\begin{gather*}}
\def\eng{\end{gather*}}
\def\bqu{\begin{question}}
\def\equ{\end{question}}

\def\im{{\rm im}}

\def\im{\imath}

\def\nm{\nonumber}

\def\ban{\begin{proof}[Answer]}
\def\ean{\end{proof}}

\def\ra{\Rightarrow}

\def\p{\partial}

\def\bqu{\begin{question}}
\def\equ{\end{question}}

\def\0110{\begin{matrix} 0 & 1\\1&0\end{matrix}}

\def\fb{\mathfrak{b}}

\def\fg{\mathfrak{g}}
\def\fh{\mathfrak{h}}

\def\fl{\mathfrak{l}}

\def\fn{\mathfrak{n}}
\def\fo{\mathfrak{o}}

\def\fs{\mathfrak{s}}

\def\ban{\begin{proof}[Answer]}
\def\ean{\end{proof}}

\def\ben{\begin{equation}}
\def\een{\end{equation}}

\def\la{\langle}
\def\ra{\rangle}

\def\j1{{(j+1)}}

\def\e{\epsilon}

\def\ad{\text{ad}}
\def\im{\text{im}}

\begin{document}

\title[Intrinsic construction of invariant functions]{Intrinsic construction of invariant functions on simple Lie algebras}

\author{Zhaohu Nie}
\email{zhaohu.nie@usu.edu}
\address{Department of Mathematics and Statistics, Utah State University, Logan, UT 84322-3900}

\date{\today}


\begin{abstract}
An algorithm for constructing primitive adjoint-invariant functions on a complex simple Lie algebra is presented. The construction is intrinsic in the sense that it does not resort to any representation. A primitive invariant function on the whole Lie algebra is obtained by lifting a coordinate function on a Kostant slice of the Lie algebra. Such an intrinsic construction of invariant functions is most useful for the bigger exceptional Lie algebras such as the $E$'s.
The Maple implementation of this algorithm is outlined at the end and will be applied to these exceptional Lie algebras in a future work. 
\end{abstract}

\maketitle

\section{Introduction}

Let $\fg$ be a complex simple Lie algebra of rank $l$ with adjoint group $G$. 
We recall that $G$ acts on $\fg$ by the adjoint action, and therefore on the algebra ${\mathcal P}(\fg)$ 
of polynomials on $\fg$ by its contragredient, that is, 
\begin{equation}\label{def G action}
(g\cdot P)(x) = P(\text{Ad}_{g^{-1}} x),\quad g\in G,\ P\in {\mathcal P}(\fg),\ x\in \fg.
\end{equation}
Let 
$$I(\fg)={\mathcal P}(\fg)^G$$
 be the algebra of polynomials on
$\fg$ invariant under the above action of $G$.  A 
well-known theorem of Chevalley \cite{C} asserts that $I(\fg)$ is polynomial algebra on $l$ 
homogeneous polynomials 
$I_1 ,\cdots,I_l$, that is, 
\begin{equation*}
I(\fg) = {\Bbb C}[I_1,\cdots,I_l].
\end{equation*}
We will refer to the $I_j$'s as
\emph{primitive invariant functions} on $\fg$.
Write the \emph{degrees}
\begin{equation}\label{degs}
\deg I_{j} = d_{j},\quad j=1,\cdots, l.
\end{equation}
We will assume that the $I_{j}$'s are ordered in the sense
that
\begin{equation*}
d_1\leq d_2\leq\cdots\leq d_l.
\end{equation*}
The numbers 
\begin{equation}\label{expo}
m_j= d_j -1,\quad j=1,\cdots, l,
\end{equation}
are called the \emph{exponents} of $\fg$. 

Although the choice of the $I_j$'s is not unique, the degrees $d_j$ and hence the exponents $m_j$ are intrinsic to $\fg$ which constitute important invariants (see \cite{ChB}). Our main objective in this paper is to give an algorithm to explicitly and  intrinsically construct a set of primitive invariant functions.
We note that our invariant functions are defined on the whole Lie algebra $\fg$. We also comment that our construction is uniform, explains the pattern for the exponents, and 
does not resort to any representation. The author has implemented his algorithm on Maple. 

The traditional way of obtaining such invariant functions is extrinsic by employing a faithful representation $\rho: \fg\to \fg\fl(V)$ of $\fg$ on a vector space $V$. Usually the first fundamental representation of the Lie algebra is used because of its small dimension. 
For $x\in \fg$, since the characteristic polynomial of $\rho(x)$ is invariant under the conjugation by $GL(V)$, its certain coefficients are the sought-after primitive invariant functions of $x$. These certain coefficients are the sums of the principal minors of the matrix $\rho(x) $ with dimensions equal to the degrees $d_j$ \er{degs} of the Lie algebra. To this author, there are several drawbacks to this approach. First, this construction uses the \emph{a prior} information of the degrees $d_j$ of $\fg$ without being able to provide any deeper reason, and in the case of $D_{2n}=\fs\fo(4n)$ where the degree $2n$ has multiplicity 2, a special formula is needed for the Pfaffian. 
Furthermore for bigger exceptional Lie algebras, their representations are hard to be made explicit, and the enormous cardinality of the principal minors of a big matrix prevents this method from being efficient. 
In particular, explicit forms of invariant functions on $E_8$ are only known up to the second one of degree 8 \cite{8E8}.   

In view of the above, an intrinsic and uniform method is clearly desirable. Intuitively speaking, our algorithm uses the restriction of the adjoint representation on a principal $\fs\fl_2$ subalgebra in $\fg$, and we gain independence from other representations and furthermore computational efficiency  in this way. In particular, our algorithm is very effective in obtaining the interesting restrictions of the invariant functions on the shifted Borel subalgebras, that is, carrying out step (i) in Theorem \ref{main}. 
 
The foundation for our construction is Kostant's profound studies \cites{K1,K2,K3} on invariant functions, which we now introduce.  Fix a Cartan subalgebra $\fh\subset \fg$. Let $\Delta$ be the corresponding root system, $\Delta_{\pm}$ a choice of positive/negative roots, and 
$\pi =\{\alpha_1,\cdots,\alpha_l\}$ the positive simple roots. Let $\fg = \fh \oplus \bigoplus_{\alpha\in \Delta} \fg_{\alpha}$
be the root space decomposition,
with $\fg_\a$ generated by a root vector $e_\a$. 
For $\a\in \Delta_+$, let $H_{\a}=[e_{\a},e_{-\a}]$, and we require $\a(H_{\a})=2$ for the choices of root vectors. For $1\leq i\leq l$, the $H_{\a_i}$ 
form a basis of $\fh$. 

The height (or the order) $o(\a)$ of a root $\alpha\in \Delta$ is defined as
\begin{equation}\label{def ht}
o(\alpha) = \sum_{i=1}^l n_i,\quad \text{if }\a=\sum_{i=1}^l n_i \alpha_{i}.
\end{equation}
This also induces a height gradation 
\begin{equation}\label{height}
\fg\cong\bigoplus_k \fg_k,\quad \fg_k=\bigoplus_{o(\a)=k} \fg_\a\text{ and } \fg_0=\fh.
\end{equation}
For $x\in \fg_k$, we write $o(x)=k$ by abusing the notation and call $k$ the height of $x$. 
Let $\fn=\bigoplus_{\a\in \Delta_+} \fg_\a=\bigoplus_{k>0}\fg_k$ be the maximal nilpotent subalgebra of $\fg$, $\fb=\fh\oplus \fn$ the Borel subalgebra, and $N$ the unipotent subgroup of $G$ corresponding to $\fn$.

Define 
\begin{equation}\label{def e0}
{\epsilon = \sum_{i=1}^l e_{-\alpha_{i}}}.
\end{equation}
Let $\fs$ be a complement of $[\e,\fg]$ in $\fg$, that is, 
\begin{equation}\label{split}
\fg\cong \fs\oplus [\e,\fg].
\end{equation}
Then by \cites{K1,K2}, $\fs\subset \fn$, and $\dim(\fs)=l$ is equal to the rank. We call $\fs$ a Kostant slice, and let $\{s_j\}_{j=1}^l$ be a homogeneous basis of $\fs$ with respect to the height gradation \er{height}. 

The following theorem summarizes several important results of Kostant on invariant functions. 
\begin{theorem}[Kostant \cites{K1,K2,K3}]\label{Kos iso} 
The heights of the $s_j$
are correspondingly the exponents of the Lie algebra $\fg$. That is, if we order the $\{s_j\}$ so that $o(s_1)\leq o(s_2)\leq \cdots \leq o(s_l)$, then 
\begin{equation}\label{exp pops}
o (s_j)=m_j, \quad 1\leq j\leq l. 
\end{equation}

Furthermore, there is a sequence of isomorphisms through restrictions
\begin{equation}\label{restriction}
I(\fg)\xrightarrow[r_1]{\cong} \cp(\e+\fb)^N\xrightarrow[r_2]{\cong} \cp(\e+\fs),
\end{equation}
where $\cp(\e+\fb)^N$ is the algebra of polynomials on $\e+\fb$ invariant under the $N$ action, and $\cp(\e+\fs)$ is the algebra of all polynomials on $\e+\fs$. 
\end{theorem}

The following is our main result. 
\begin{theorem}\label{main}
\begin{enumerate}
\item
There is an explicit algorithm for constructing the inverse
to $r_2$ in \er{restriction}:
$$r_2^{-1}: \cp(\e+\fs)\to \cp(\e+\fb)^N.$$

More precisely, let the $\xi_j$ be the coordinates of a general point 
$$\e+\sum_{j=1}^l \xi_j s_j\in \e+\fs.$$ 
Then
there is an algorithm for constructing 
$l$ primitive invariant function $I_j$ defined on $\e+\fb$ of degree $d_j$ such that 
\begin{equation}\label{strong}
I_j(\e+\sum_{i=1}^l \xi_i s_i)=\xi_j,\quad 1\leq j\leq l. 
\end{equation}

\item 
Furthermore, there is an explicit algorithm for constructing the inverse to $r_1$ in \er{restriction} such that the invariant functions are defined on the whole $\fg$. 
\end{enumerate}
\end{theorem}

We present our basic setup and compution techniques in Section 2. In Section 3, we present our algorithms in the proofs of the two parts of our main Theorem \ref{main} together with several propositions. In part (i), our algorithm for constructing $r_2^{-1}$ lifts
the values of the invariant functions from the slice $\e+\fs$ to $\e+\fb$ and then to the Cartan subalgebra $\fh$. In Proposition \ref{weyl inv}, we prove the Weyl invariance of the resulted function on $\fh$. 
In part (ii), we present the similar algorithm for constructing $r_1^{-1}$ to define the invariant functions on $\fg$. The complexity of this step is much bigger than the previous step, and to a large extent accounts for the difficulty in getting the invariant functions on the whole Lie algebra.   
In Section 4, we outline how the author has implemented the 
algorithm on the software Maple. 

Because of the special roles played by exceptional Lie algebras in mathematics and physics, we expect our algorithm and the explicit invariant functions it produces to have applications in a range of areas such as integrable systems and higher Casimir operators. For example, to get all the first integrals for a full Toda flow \cite{GS} on an exceptional Lie algebra, one needs certain portions of such explicit invariant functions.

\medskip

\noindent{\bf Acknowledgment.} The author thanks Ian Anderson for help with Maple and Luen-Chau Li for interest in this work. He also thanks the referee for suggestions which improve the exposition of the paper. 

\section{Setup and computation techniques}
Our later calculations rely on the following setup of Kostant \cite{K2} in an essential way. We have previously used such a setup extensively and developed some techniques for computation in \cite{LN}. This paper is a further application of such techniques. 

To begin with, we recall that the polynomial algebra ${\mathcal P}(\fg)$ can be identified with the symmetric algebra
$S = S(\fg^*)$ on $\fg^*$, the dual of $\fg$.
On the other hand, we can associate to each $x\in \fg$ a differential
operator $\partial_{x}$ on $\fg,$ defined by 
\begin{equation}\label{p x}
(\partial_{x}f)(y) = \frac{d}{dt}{\Big|_{t=0}} f(y + tx), \quad f\in
C^{\infty}(\fg).
\end{equation}
In this way we have a linear map $x\mapsto \partial_{x}$ which can be
extended to an isomorphism from the symmetric algebra
$S_{*} = S(\fg)$ on $\fg$ to the algebra of differential operators
$\partial$ with constant coefficients on $\fg.$  From now onwards we
will identify these two spaces. With this identification,
we have a nondegenerate pairing between $S_{*}$ and $S$
given by
\begin{equation}\label{pairing}
\langle \partial,f\rangle=(\partial f)(0),\quad \partial\in S_{*},\ f\in S,
\end{equation}
where $(\partial f) (0)$ denotes
the value of the function $\partial f$ at $0\in \fg.$
It is clear that both $S_*$ and $S$ are graded from the tensor structure:
$S_{*} = \oplus_{k\geq 0} S_*^k$, $S =\oplus_{k\geq 0} S^{k},$ 
and $S_*^j$ pairs nontrivially only with $S^j$. 

If $f\in S^k$ and $x\in \fg,$ it follows from the Taylor expansion that
\begin{equation}\label{Taylor}
\Big\langle \frac{\p_x^k}{k!},f\Big\rangle=f(x).
\end{equation}

It is clear that the adjoint action of $G$ on $\fg$ can  be naturally extended to an action of $G$ on $S_{*}.$ On the other hand, $S$ is a $G$-module as its contragredient by \er{def G action}. (We denote the actions of $G$ and later of $\fg$ by a dot.) We have 
\begin{equation}\label{inv under G}
\langle g\cdot\partial, g\cdot f\rangle = \langle \partial, f\rangle,\quad \forall g\in G,\ \partial \in S_{*},\ f\in S.
\end{equation}
By differentiation, $S$ and $S^*$ become $\fg$-modules and
the actions of $\fg$ on both spaces are by derivations. Therefore we have the following properties:
\begin{align}
\label{homom}
&[x,y]\cdot \partial=x\cdot(y\cdot \p)-y\cdot(x\cdot \p),\quad x,y\in \fg,\ \p\in S_*, &&\text{Lie alg hom}\\
\label{action}
&x\cdot \partial_{y} = \partial_{[x,y]},\quad x,\, y\in \fg, &&\text{adj action}\\
\label{prod rule}
&x\cdot (\partial\delta)=(x\cdot \partial)\delta+\partial(x\cdot \delta),\quad x\in \fg,\, \partial, \delta\in S_*,&&\text{derivatioin}\\
\label{power rule}
&x\cdot\partial^n=n \partial^{n-1}(x\cdot \partial),&&\text{power rule}\\
\label{inv property}
&f\in I(\fg)\Longrightarrow x\cdot f=0,\quad \forall x\in \fg.&&\text{inv property}
\end{align}

Since the pairing between $S_{*}$ and $S$ obeys \er{inv under G},
it follows from derivation that the $\fg$-actions satisfy
\begin{equation*}
\langle x\cdot \partial,f\rangle+\langle \partial,x\cdot f\rangle=0,\quad \forall x\in \fg,\ \p\in S_*,\ f\in S.
\end{equation*}
This and \er{inv property} imply that
\begin{equation}\label{source}
\langle x\cdot \partial,f\rangle=0,\quad \forall f\in I(\fg),\, x\in \fg.
\end{equation}

There is a grading element $x_0\in \fh$ defined by the conditions that 
\begin{equation}\label{def x0}
\a_i(x_0)=1,\quad \forall\, 1\leq i\leq l. 
\end{equation}
By \er{def x0} and \er{def ht}, $\a(x_0)=o(\a)$, and $[x_0,e_\a]=o(\a)e_\a$. Thus the graded subspaces from \er{height} are $\fg_k=\{x\in \fg \,\big|\, [x_0,x]=kx\}$. This motivate the following definition of the weight structure of \cite{K1} on $S_*$. For each $k\in \Bbb Z,$ define
\begin{equation}\label{components}
\begin{split}
S_k& =\{\partial\in S_*\,\big |\,x_0\cdot \partial=k\partial\}.
\end{split}
\end{equation}

Applying \er{source} to $x_0\in \fh$ gives us the first vanishing result of Kostant \cite{K2}. If 
$\partial\in S_k$ for $k\neq 0$, then by \er{components}, 
$\partial=\frac{1}{k} x_0\cdot \partial$. 
Then in view of \er{source}, we have
\begin{equation}\label{weight van}
\partial\in S_k\text{ for }k\neq 0\Longrightarrow
\langle \partial,f\rangle =0\,\,\,\,\hbox{for all}\,\,f\in I(\fg).
\end{equation}

Applying \er{source} to a general $x\in \fh$ gives us the following refined vanishing. 
\begin{lemma}[{\cite{K3},\cite{LN}*{Lemma 3.1}}]\label{use fh}  
For all $f\in I(\fg),$ $p\in \fh,$
$$
\left< \partial^{n}_{p}\prod_{\a\in \Delta} \partial^{m_{\a}} _{e_{\a}}, f
\right>  =0
$$
unless $\sum_{\a\in X} m_{\a}\, \a =0.$
\end{lemma}

In our later constructions, we will further exploit \er{source} by applying it to other elements in $\fg$. 
The most convenient formulation for us is the following ``integration by parts" formula to ``move things around."
\begin{lemma}[\cite{LN}*{Lemma 6.2}]\label{int by parts} Let  $f\in I(\fg),$  then for all $x,y\in \fg$, $\p\in S_*$, 
and $m\geq 0,$ we have
\begin{equation}\label{gen ibp}
\langle \partial_x^m\partial_{[x,y]}\p,f\rangle
=\frac{1}{m+1}\big\la \partial_x^{m+1}(y\cdot \p),f\big\ra.
\end{equation}
\end{lemma}

For the reader's convenience, we repeat the brief proof. 
\begin{proof}
By using \er{action}, \er{power rule}, \er{prod rule} and \er{source}, we
find that
\begin{align*}
\langle \partial_x^{m} \partial_{[x,y]}\partial, f\rangle
&=-\langle \partial_x^{m} (y\cdot \partial_{x})\partial, f\rangle\\
&= -\frac{1}{m+1} \langle (y\cdot \partial_x^{m+1})\partial,f\rangle\\
&=-\frac{1}{m+1} \langle y\cdot (\partial_x^{m+1}\partial),f\rangle
   +\frac{1}{m+1} \langle \partial_x^{m+1}(y\cdot \partial),f\rangle\\
&=\frac{1}{m+1} \big\langle \partial_x^{m+1}(y\cdot \partial),f\big\rangle.
\end{align*}
\end{proof}

\section{Constructive proof of the main theorem}

In this section, we present our algorithms in the forms of proofs to the two parts of our main Theorem \ref{main}, and we present several supporting propositions. In view of \er{Taylor} and the multinomial theorem, to obtain a function $f(x)$ we can choose a basis for $\fg$ and compute the derivatives \er{pairing} where $\partial$ is a differentiation operator constructed from the basis vectors. See the end of this section for how to assemble the function. The most natural basis for $\fg$ is the one of root vectors, for example a Chevalley basis, and we will express our final result in this basis. But in the course of the algorithm, we need to use a different basis of $\fg$, which we define first. This basis is crucial to our inductive procedures later.
Define
\begin{equation}\label{ujk}
s_j^k=(\ad_\e)^k s_j,\quad 0\leq k\leq 2m_j,\ 1\leq j\leq l.
\end{equation}
Then $s_j^k=[\e,s_j^{k-1}]$ for $k\geq 1$. 

\begin{lemma}\label{base change}
$\{s_j^k\}_{1\leq j\leq l}^{0\leq k\leq 2m_j}$ \text{is a basis of } $\fg$.
\end{lemma}

\begin{proof}
This follows from \er{split} and Kostant's work in \cite{K1}.
\end{proof}

The height of $s_j^k$
is $m_j-k$ by \er{exp pops} and \er{def e0}, and therefore
\begin{equation}\label{basis}
\cu:=\{s^k_j\}_{0\leq j\leq l}^{0\leq k\leq m_j-1}\ \text{ is a basis of }\fn=\bigoplus_{k> 0}\fg_k.
\end{equation}
We will denote a general element of $\cu$ by $u$. Note that for such $u$'s, either $u\in \fs$ or $u\in [\e, \fg]$ in the decomposition \er{split}. (Actually the preimage of such a $u\in  \im(\ad_\e)$ is unique since $\ad_\e$ has no kernel in $\fb$ \cite{K1}.)

Similarly we have that 
\begin{equation}\label{neg basis}
\cw:=\{s_j^k\}_{1\leq j\leq l}^{m_j+2\leq k\leq 2m_j}\ \text{ is a basis of }\bigoplus_{k\leq -2} \fg_k. 
\end{equation}
We will denote a general element of $\cw$ by $w$. Note that all such $w$'s belong to $[\e, \fg]$ in \er{split}. 

By the special nature of $\fg_0=\fh$ and $\fg_{-1}$, this author will use their natural basis $\{H_{\a_i}\}_{1\leq j\leq l}$ and $\{e_{-\a_i}\}_{1\leq j\leq l}$ in the presentation of this paper. 

Therefore we set out to compute various derivatives of the invariant function using the new bases $\cu$ and $\cw$ inductively. 

\begin{proof}[Proof of Theorem \ref{main} (i)]  
Since we want to construct invariant functions, we will enforce 
the invariance property \er{source} or more explicitly Lemma \ref{int by parts}. In this part, we will inductively show that this rule and the condition \er{strong} determine the invariant function $I_j$ on $\e+\fb$. 
Then we prove that the constructed function when restricted on $\fh$ is {invariant under the Weyl group} in Propositions \ref{weyl inv}. 
Proposition \ref{indep} quickly shows that such functions are algebraically independent. 

In this proof, we will work with a fixed $I_j$, and we often write $I$ for short. Also $d=d_j$ and $m=m_j$. 

Using \er{Taylor} and the multinomial theorem, to describe $I\in \cp(\e+\fb)$, we need to determine all the following polynomials
\begin{equation}\label{mg}
\la \p_U\p_\e^b\p_p^a,I\ra=\la \p_{u(1)}\p_{u(2)}\cdots\p_{u(c)}\p_\e^b\p_p^a,I\ra
\end{equation}
of degree $a$ in $p\in \fh$. 
Here $U=(u(1),\cdots,u(c))$ is a sequence with each $u(i)$ (possibly repeating) from $\cu$ in \er{basis}. 
The expression \er{mg} is nonzero only if 
\begin{equation}
\label{first non van}
\begin{split}
a+b+c&= d,\\
\sum_{i=1}^c o(u(i))&=b,
\end{split}
\end{equation}
by \er{pairing} and \er{weight van}. (This is related to the $x_0$-grading by Kahzdan in \cite{K3}.) 

We run increasing induction on $a$ and decreasing induction on $b$ to define the polynomials in \er{mg}.

The first case is $a=0$ and $b=d-1=m$. Then we need to determine all the 
$\la \p_u\p_\e^m,I\ra$ with $o(u)=m$. For $u=s_j$ where $j$ is our fixed index, 
applying the multinomial theorem and the vanishing results, we have 
\begin{align}
I_j(\e+s_j) & = \frac{1}{d!} \bigl\la (\p_{s_j}+\p_\e)^{d},I_j\bigr\ra && \text{by }\er{Taylor}\nm\\
& = \frac{1}{m!}\la \p_{s_j}\p_\e^{m},I_j\ra. && \text{by }\er{weight van}
\label{neater}
\end{align}
Therefore the defining condition \er{strong} in this case, $I_j(\e+s_j)=1$, implies
\begin{equation}\label{use fctl}
\la \p_{s_j}\p_\e^{m},I_j\ra=m!.
\end{equation}

For $u=s_i\in \fs$ with $o(s_i)=m$ but $i\neq j$ (hence the multiplicity of $m$ as an exponent is at least 2 and this happens, among all the simple Lie algebras, only for $D_{2n}$ and $m=2n-1$, $n\geq 2$), similarly to \er{neater}, the defining condition \er{strong} in this case, $I_j(\e+s_i)=0$, implies 
\begin{equation}\label{more van 1}
\la \p_{s_i}\p_\e^m,I_j\ra=0,\quad o(s_i)=o(s_j),\ i\neq j.
\end{equation}
For $u=[\e,v]\in [\e, \fg]$, the vanishing property \er{source}, together with \er{action} and \er{power rule}, forces 
\begin{equation}\label{just put in}
\la \p_u\p_\e^m,I\ra=\la \p_{[\e,v]}\p_\e^m,I\ra=-\frac{1}{m+1}\la v\cdot \p_\e^{m+1},I\ra=0.
\end{equation}

Now we compute \er{mg} for $a=0$ and all $b$. (Such expressions for \er{mg} are all numbers.) If in $U$, at least one $u(i)\in [\e, \fg]$, say $u(1)=[\e,v_1]$, then by Lemma \ref{int by parts} we have, with 
$\tilde U=(u(2),\cdots,u(c))$, 
\begin{equation}\label{quicker}
\begin{split}
\la \p_U\p_\e^b,I\ra &= \la \p_{u(1)}\p_{\tilde U}
\p_\e^b,I\ra=\la \p_{[\e,v_1]}\p_{\tilde U}
\p_\e^b,I\ra=\frac{1}{b+1}\Big\la\p_\e^{b+1} \big(v_1\cdot(\p_{\tilde U})\big),I\Big\ra\\
&=\frac{1}{b+1}\Big(\sum_{n=2}^c \la \p_\e^{b+1}\p_{u(2)}\cdots\p_{[v_1,u(n)]}\cdots\p_{u(c)},I\ra\Big),
\end{split}
\end{equation}
where all the terms on the right have $(b+1)$ $\e$'s, and hence are known from the induction hypothesis.
We need to show {compatibility} when there are two $u(i)\in [\e, \fg]$, and this is done in Proposition \ref{compatibility}. 

On the other hand, \er{strong} implies 
\begin{equation}\label{more van 2}
\la \p_U\p_\e^b,I\ra=0,\quad \text{if }c\geq 2\text{ and all the }u(i)=s_k\in \fs.
\end{equation} 

Now assume that the \er{mg} have been computed when the degree in $p$ is $\leq a-1$ with $a\geq 1$, and we compute it for degree $a$. 

We in effect use the fact that every $p\in \fh$ is in $[\e, \fg]$ in \er{split}. Actually for $p=\sum_{i=1}^l p_i H_{\a_i}$, define 
\begin{equation}\label{def xp}
x_p=\sum_{i=1}^l p_i e_{\a_i},\text{ then }-[\e,x_p]=p.
\end{equation}
Here $x_p$ can be regarded a linear function in $p$ with values in $S_*^1$.

Then Lemma \ref{int by parts} gives  
\begin{equation}\label{a>0}
\begin{split}
\la \p_U\p_\e^b\p_p^{a},I\ra&=-\la \p_U\p_\e^b\p_p^{a-1}\p_{[\e,x_p]},I\ra\\
&=\frac{1}{b+1}\sum_{n=1}^c \la \p_\e^{b+1}\p_{u(1)}\cdots\p_{[u(n),x_p]}\cdots\p_{u(c)}\p_p^{a-1},I\ra\\
&\quad\ +\frac{a-1}{b+1}\la \p_\e^{b+1}\p_{U}\p_{[p,x_p]}\p_p^{a-2},I\ra\\
&=\frac{1}{b+1}\sum_{i=1}^l\sum_{n=1}^{c}p_i \la \p_\e^{b+1}\p_{u(1)}\cdots\p_{[u(n),e_{\a_i}]}\cdots\p_{u(c)}\p_p^{a-1},I\ra\\
&\quad\ +\frac{a-1}{b+1}\sum_{i=1}^l p_i\a_i(p) \la \p_\e^{b+1}\p_{U}\p_{e_{\a_i}}\p_p^{a-2},I\ra,
\end{split}
\end{equation}
by \er{def xp} and hence 
$$
[p,x_p]=\sum_{i=1}^l p_i \a_i(p) e_{\a_i},
$$
which has degree 2 in $p$. The factors in the first sum have degrees $a-1$ in $p$, and the factors in the second sum have degrees $a-2$ in $p$. Using the induction hypothesis and \er{basis}, all such factors can be expressed in the known cases of \er{mg}. 

We can continue \er{a>0} all the way until we get $a=d$, where we have
\begin{equation}\label{a=d}
\begin{split}
\la \p_p^d, I\ra&=(d-1)\sum_{i=1}^l p_i\a_i(p) \la \p_\e\p_{e_{\a_i}}\p_p^{d-2},I\ra.
\end{split}
\end{equation}
Then by \er{Taylor}, our function $I(p)$ on $\fh$ is 
\begin{equation}\label{def ip}
I(p)=\frac{1}{d!}\la \p_p^d,I\ra.
\end{equation}

Proposotions \ref{weyl inv} and \ref{indep} below prove that the $\{I_j\}_{j=1}^l$ constructed this way are 
algebraically independent and invariant under the Weyl group when restricted to the Cartan subalgebra $\fh$. 
\end{proof}

\begin{proposition}\label{compatibility} There is compatibility when there are two choices for $u\in [\e, \fg]$ in \er{quicker}. 
\end{proposition}

\begin{proof} Assume, for example, $u(1)=[\e,v_1],u(2)=[\e,v_2]$. Then, with $U'=\\(u(3),\cdots,u(c))$, 
$$
\la \p_U\p_\e^b,I\ra=\la \p_{u(1)}\p_{u(2)}\p_{U'}\p_\e^b,I\ra
$$
can be computed in two ways using $v_1$ or $v_2$ in \er{quicker}. The first answer $A_1$ using $v_1$ is, by \er{prod rule} and \er{action},  
\begin{align*}
A_1&=\frac{1}{b+1}\Big\la\p_\e^{b+1} \big(v_1\cdot(\p_{u(2)}\p_{U'})\big),I\Big\ra\\
&=\frac{1}{b+1}\Big(\la\p_\e^{b+1} \p_{[v_1,u(2)]}\p_{U'},I\ra+\la\p_\e^{b+1} \p_{u(2)}(v_1\cdot \p_{U'}),I\ra\Big)
\end{align*}
and similarly for the second answer $A_2$ using $v_2$. Therefore for the difference, we have 
\begin{align*}
&\quad (b+1)(A_1-A_2)\\
&=\big\la\p_\e^{b+1} \p_{([v_1,u(2)]-[v_2,u(1)])}\p_{U'},I\big\ra+\big\la\p_\e^{b+1} \p_{u(2)}(v_1\cdot \p_{U'}),I\big\ra-\big\la\p_\e^{b+1} \p_{u(1)}(v_2\cdot \p_{U'}),I\big\ra\\
&= \big\la\p_\e^{b+1} \p_{[\e,[v_1,v_2]]}\p_{U'},I\big\ra+\big\la\p_\e^{b+1} \p_{[\e,v_2]}(v_1\cdot \p_{U'}),I\big\ra-\big\la\p_\e^{b+1} \p_{[\e,v_1]}(v_2\cdot \p_{U'}),I\big\ra\\
&=\frac{1}{b+2}\bigg(\Big\la\p_\e^{b+2} \big({[v_1,v_2]}\cdot\p_{U'}\big),I\Big\ra+\Big\la\p_\e^{b+2}\big({v_2}\cdot (v_1\cdot \p_{U'})\big),I\Big\ra\\
&\qquad\qquad -\Big\la\p_\e^{b+2} \big({v_1}\cdot (v_2\cdot \p_{U'})\big),I\Big\ra\bigg)\\
&=0,
\end{align*}
where the first term in the second equality uses the Jacobi identity
$$
[v_1,u(2)]-[v_2,u(1)]=[v_1,[\e,v_2]]+[[\e,v_1],v_2]=[\e,[v_1,v_2]],
$$
the third equality uses Lemma \ref{int by parts} again, and the last identity uses the Lie algebra homomorphism property \er{homom} (with its root in the Jacobi identity). 
\end{proof}

\begin{remark} Furthermore when $a\geq 1$, if in $U$ there exists a $u\in [\e, \fg]$, there is an alternative approach similar to \er{quicker}, which is compatible with \er{a>0}, by the same reason as above. 
\end{remark}

\begin{proposition}\label{weyl inv} The $I(p)$ on $\fh$ defined in \er{def ip} is invariant under the Weyl group $W$. 
\end{proposition}
\begin{proof}
Since $W$ is generated on $\fh$ by simple reflections $r_i$ through the hyperplanes defined by $\a_i=0$ for $1\leq i\leq l$, we only need to prove the invariance of the function $I(p)$ under $r_i$ for any $i$. Fix an $i$ and we omit it from the notation. 

We use an {\bf orthogonal} basis of $\fh$ with the first vector being $H_\a=H_{\a_i}$. Then we write 
\begin{equation}\label{orth decom}
p=x H_{\a}+Y,\quad Y\perp H_{\a}\Longleftrightarrow \a(Y)=0.
\end{equation}
In this orthogonal basis, the reflection $r_i$ is just the transformation $x\mapsto -x$, and we only need to prove the $\la \p_p^d,I\ra$ in \er{a=d} is a function of $x^2$. For  that purpose we run decreasing induction on $k$ and $l$ to prove that in general the 
$$
D(k,l):=\la \p_{e_\a}^k \p_{e_{-\a}}^k \p_Y^l \p_p^{d-2k-l}, I\ra
$$
are functions of $x^2$, with $\la \p_p^d,I\ra=D(0,0)$. 

When $2k+l>d$, $D(k,l)=0$ by \er{pairing}. When $2k+l=d$, $D(k,l)=\la \p_{e_\a}^k \p_{e_{-\a}}^k \p_Y^l, I\ra$ are constants with respect to $x$, since $Y$ doesn't involve $x$. 

Now by Lemma \ref{int by parts} and using \er{orth decom}, we have, for $d-2k-l\geq 1$, 
\begin{align*}
D(k,l)&=\la \p_{e_\a}^k \p_{e_{-\a}}^k \p_Y^l \p_p^{d-2k-l-1}(x\p_{H_\a}+\p_Y), I\ra\\
&= x\big\la \p_{H_\a}\p_{e_\a}^k \p_{e_{-\a}}^k \p_Y^l \p_p^{d-2k-l-1},I\big\ra+\big\la \p_{e_\a}^k \p_{e_{-\a}}^k \p_Y^{l+1} \p_p^{d-2k-l-1},I\big\ra\\
&= x\big\la \p_{[e_\a,e_{-a}]}\p_{e_\a}^k \p_{e_{-\a}}^k \p_Y^l \p_p^{d-2k-l-1},I\big\ra+D(k,l+1)\\
&=-x\frac{1}{k+1}\Big\la \p_{e_{-\a}}^{k+1} \big(e_\a\cdot (\p_{e_\a}^k \p_Y^l\p_p^{d-2k-l-1})\big),I\Big\ra+D(k,l+1)\\
&=\frac{d-2k-l-1}{k+1}x\a(p)\big\la \p_{e_\a}^{k+1}\p_{e_{-\a}}^{k+1}\p_Y^l\p_p^{d-2k-l-2},I\big\ra+ D(k,l+1)\\
&=\frac{2(d-2k-l-1)}{k+1}x^2 D(k+1,l)+D(k,l+1),
\end{align*}
since 
\begin{align*}
&\quad -e_\a\cdot (\p_{e_\a}^k \p_Y^l\p_p^{d-2k-l-1})\\
&=l\p_{e_\a}^k \p_Y^{l-1}\p_{[Y,e_\a]}\p_p^{d-2k-l-1}+(d-2k-l-1)\p_{e_\a}^k \p_Y^{l}\p_p^{d-2k-l-2}\p_{[p,e_\a]}\\
&=(d-2k-l-1)\a(p)\p_{e_\a}^{k+1} \p_Y^{l}\p_p^{d-2k-l-2}
\end{align*}
due to that $[Y,e_\a]=\a(Y)e_\a=0$, $[p,e_\a]=\a(p)e_\a$, and $\a(p)=2x$  by \er{orth decom}. 
Therefore the appearance of $x$ in $D(k,l)$ is always through an $x^2$ entry. 
\end{proof} 

\begin{proposition}\label{indep} The $\{I_j\}_{j=1}^l$ are algebraically independent. 
\end{proposition}

\begin{proof} This is clear from our defining condition \er{strong}, since the $I_j$ restrict to the coordinates $\xi_j$ on the slice $\e+\fs$. 
\end{proof}

\begin{remark} In a sense, the above algorithm is the reversion of the procedures in \cite{LN}*{\S 6}. Here we start with a high root vector and push the function down to $\fh$. In \cite{LN} we derived information about higher and higher root vectors starting from some knowledge on $\fh$. 
The direction here is more delicate.
\end{remark}

\begin{proof}[Proof of Theorem \ref{main} (ii)] 
The further lifting of the invariant function $I=I_j$ to the whole Lie algebra $\fg$ involves considering all such terms
\begin{equation}\label{mmg}
\la\p_W\p_\e^b\p_p^a\p_U,I\ra
\end{equation}
where $\p_W=\p_{w(1)}\cdots\p_{w(\beta)}$ with each $w(i)$ from $\cw$ in \er{neg basis}, and $\p_U$ is the same as in \er{mg}. The absolute value of the total weight of $\p_W$ is 
\begin{equation}\label{ow}
-o(W)=-\sum_{i=1}^\beta o(w_i).
\end{equation}
Similarly to \er{first non van}, \er{mmg} is nonzero only if 
\begin{equation}
\label{second non van}
\begin{split}
\beta+a+b+c=d,\\
-o(W) + b=\sum_{i=1}^c o(u(i)).
\end{split}
\end{equation}  

Note that we do not need any vectors with heights $0$ or $-1$ in \er{mmg}, since all such vectors are accounted for by the $\p_p^a$ and $\p_\e^b$ terms using the following Lemma.

\begin{lemma}\label{dwp and dwep} 
If $a>0$ and one copy of the $p\in \fh$ is replaced by $H_{\a_i}$, then 
\begin{equation}\label{dwp}
\la\p_W\p_\e^b\p_p^{a-1}\p_{H_{\a_i}}\p_U,I\ra
=\frac{1}{a}\, \p_{H_{\a_i}}\big(\la\p_W\p_\e^b\p_p^a\p_U,I\ra\big).
\end{equation}

If $b>0$ and one copy of the $\e$ is replaced by $e_{-\a_i}$, then
\begin{equation}\label{dwep}
\la \p_W\p_\e^{b-1}\p_{e_{-\a_i}}\p_p^{a} \p_U,I\ra=\frac{1}{b} \Big\la\p_\e^{b}\p_p^{a} \big(x_i\cdot (\p_W\p_U)\big),I\Big\ra,
\end{equation}
where $x_i\in \fh$ is the grading element for $\a_i$ specified by the conditions that 
\begin{equation}\label{ith grading}
\a_j(x_i)= \delta_{ji}, \quad j=1,\dots,l.
\end{equation}
\end{lemma}

\begin{proof}[Proof of Lemma \ref{dwp and dwep}] 
Let $p=\sum_{i=1}^l p_i H_{\a_i}\in \fh$. The $\p_{H_{\a_i}}$ on the right hand side of \er{dwp} stands for $\frac{\p}{\p p_i}$. By the definition \er{def e0} and the conditions \er{ith grading}, we have
$$
[\e, x_i] = \sum_{j=1}^l \a_j(x_i) e_{-\a_j} = e_{-\a_i}.
$$
Then the two formulas are easy consequences of the Taylor expansion \er{Taylor} and Lemma \ref{int by parts}. 
\end{proof}

Now returning to the main proof. 
We run induction on the lexicographical order of the pair $(\beta, -o(W))$, with $-o(W)$ defined in \er{ow}, to determine such terms in \er{mmg}. 
Since $w(1)\in [\e, \fg]$ from \er{neg basis}, assume 
\begin{equation}\label{wv}
w(1)=[\e,v_1].
\end{equation} 
Then $v_1\in \bigoplus_{k\leq -1}\fg_k$. 
With $\tilde W=(w(2),\cdots,w(\beta))$, 
Lemma \ref{int by parts} gives
\begin{equation}\label{neg induct}
\begin{split}
\la \p_W\p_\e^b\p_p^{a}\p_U,I\ra&=\la \p_{[\e,v_1]}\p_{\tilde W}\p_\e^b\p_p^{a}\p_U,I\ra\\
&=\frac{1}{b+1}\sum_{m=2}^\beta \la \p_{w(2)}\cdots\p_{[v_1,w(m)]}\cdots\p_{w(\beta)}\p_\e^{b+1}\p_p^{a}\p_U,I\ra\\
&\quad\ +\frac{a}{b+1}\la \p_{\tilde W}\p_\e^{b+1}\p_p^{a-1}\p_{[v_1,p]}\p_U,I\ra\\
&\quad\ +\frac{1}{b+1}\sum_{n=1}^c \la \p_{\tilde W}\p_\e^{b+1}\p_p^{a}\p_{u(1)}\cdots\p_{[v_1,u(n)]}\cdots\p_{u(c)},I\ra.
\end{split}
\end{equation}
Here all the summands in the first sum can be expressed by \er{mmg} with $\beta-1$ elements from $\cw$ in \er{neg basis}. 
The summands in the second and the third sums can either be expressed by \er{mmg} with $\beta-1$ elements from $\cw$ if the heights of $[v_1,p]$ or $[v_1,u(n)]$ are $\geq -1$ by Lemma \ref{dwp and dwep}, or with $\beta$ elements from $\cw$ otherwise. But the new $-o([v_1,p])$ or $-o([v_1,u(n)])$ is strictly less than the old $-o(w(1))$, since the height of $v_1$ is one bigger than that of $w(1)$ in view of \er{wv}, and $p$ and $u(n)$ have nonnegative heights. Therefore the new total $-o(W)$ is strictly less than the old one. 

Therefore through this hierarchy of induction hypothesis, all the terms on the right are known. 

We note that the outcome of \er{neg induct} does not depend on the choice of $v_1$ in \er{wv}, which may not be unique. Say $v_1'=v_1+v_0$ with $[\e,v_0]=0$. Then the outcome of \er{neg induct} is linear in $v_1$ and $v_0$, and the terms for $v_0$ combine to give $\la \p_{[\e,v_0]}\p_{\tilde W}\p_\e^b\p_p^{a}\p_U,I\ra=0$ by tracing the identity backward. 
\end{proof}

After all these coefficients in \er{mg} and \er{mmg}, as functions on $p\in \fh$, are calculated, we can assembel our function as follows. Let 
$$x=\sum_{w_i\in \cw} {z_i} w_i + \epsilon + p + \sum_{u_j\in \cu} {y_j} u_j$$ 
be an element in $\fg$ with the $w_i$ from $\cw$ in \er{neg basis}, the $u_j$ from $\cu$ in \er{basis}, $p\in \fh$, and the $z_i$ and $y_j$ as coefficients. Then we get $I(x)$ by \er{Taylor}, the multinomial theorem, and the coefficients \er{mg} and \er{mmg}. 

If we change the basis back to the usual root vectors, then we get $I(x)$ for $x=p+\e+\sum_{o(\a)\neq -1} x_\a e_\a$. Using Lemma \ref{use fh}, we can 
further spell out the dependence on the $e_{-\a_i}$. At the end, we obtain the function $I(x)$ expressed in the coordinates of a general element in $\fg$:
$$x=\sum_{i=1}^l p_i H_{\a_i} + \sum_{\a\in \Delta} x_\a e_\a.$$

\section{Implementation of the algorithm on Maple}

It turns out that our algorithm is very ready for implementation on Maple, especially using the {\tt LieAlgebras} package under Maple written by Prof. Ian Anderson. One particularly useful feature is that we can do the change of basis in Lemma \ref{base change} easily. This author has written a Maple program containing all the implementations. Together with his collaborator, the author plans to apply his Maple implementation of this algorithm to the invariant functions on the Lie algebras of type $E$ and to make the results available online at the DifferentialGeometry Software Project website at the Digital Commons of the Utah State Univeristy (http://digitalcommons.usu.edu/dg/). 

In this section, we illustrate our Maple implementation using the degree 6 invariant function on $\fg_2$ for concreteness. We will also comment on the running time for other bigger examples. 

We use the basis of $\fg_2$ as made explicit in the Appendix of \cite{W-sym}. We setup our $\fg_2$ with basis 
\begin{align*}
&e_1= H_{\a_1} &&e_2=H_{\a_2} && \\
&e_3= e_{\a_1} && e_4=e_{\a_2} && e_5=e_{\a_1+\a_2}\\
&e_6=e_{2\a_1+\a_2} && e_7=e_{3\a_1+\a_2} && e_8=e_{3\a_1+2\a_2}\\
&e_9= e_{-\a_1} && e_{10}=e_{-\a_2} && e_{11}=e_{-\a_1-\a_2}\\
&e_{12}=e_{-2\a_1-\a_2} && e_{13}=e_{-3\a_1-\a_2} && e_{14}=e_{-3\a_1-2\a_2}
\end{align*}

We choose our slice elements in \er{split} to be $s_1=e_4$ and $s_2=e_8$. Let $\e=e_9+e_{10}$, and we do the change of basis in Lemma \ref{base change}. Order the new basis according to the heights following the above pattern, and denote them by $\{f_i\}_{i=1}^{14}$. We also record where we have the relation $u=[\epsilon, v]$ in \er{ujk} in a table.  

The degree $d$ in this example is set to be 6. A nonzero term from \er{mmg}, for example, 
\begin{equation}\label{record}
\la \p_{f_{11}}^2\p_{f_8}\p_\e \p_p^2, I\ra\text{ is recorded by }y_{11}^2 y_8\text{ with }b=1\text{ and }a=2.
\end{equation}
At some point of the program, this derivative function is calculated as $3136\,( 3p_2-p_2)( 3p_1-2p_2)$. We record such information in the table ${\tt valuedata}$ as 
$$
{\tt valuedata}[y_{11}^2 y_8] = 3136\,( 3p_2-p_2)( 3p_1-2p_2).
$$

We generate the possible nonzero terms according to the condition \er{second non van}, and we can order them according to our induction order. There are 18 terms of the form \er{mg} with no $\p_W$, 
and 535 terms of the form \er{mmg} in general. 

To start the induction, we input the first few terms \er{use fctl}, \er{more van 1}, \er{more van 2}, and we choose to input also \er{just put in}. 

Then we compute the other terms of the form \er{mg} by formulas \er{quicker} when $a=0$ and \er{a>0} when $a>0$. We also calculate the purely Cartan term by \er{a=d}. 

Finally, and this is the big step, we compute the terms of the form \er{mmg} by formula \er{neg induct}, incorporating the two formulas \er{dwp} and \er{dwep} when $H_{\a_i}$ or $e_{-\a_i}$ appears. 

Using the procedure described at the end of Section 3, we get a function for $x\in \fg_2$ in the original root vector basis. It turns out to be one quarter of the 
sum of principal minors of dimension 6 of 
the corresponding matrix representation of $\fg_2$. The whole procedure takes about 8 seconds on a usual laptop. 

This author has tried his Maple program for other bigger invariant functions. For the Pfaffian of degree 5 on $D_5$, there are 51 terms of the form \er{mg} and 34366 terms of the form \er{mmg}. The whole calculation takes about one hour on a usual laptop. The author has also calculated the Pfaffian using a simple implementation of the definition and that calculation actually took slightly longer than one hour. The two results of course exactly match (up to a sign). 

The author has also tried his program for the second invariant function of degree 5 on $E_6$. He obtained the structure constants of $E_6$ following \cite{vavilov}. There are 72 terms of the form \er{mg} and 452056 terms of the form \er{mmg}. 
The author estimates that it will take about one day on a usual laptop to fully calculate the invariant function of degree 5 on $E_6$. 
He would like to remark that the program is very stable while running through the possible terms, and calculations of such magnitude should be carried out on a more powerful station or using a more efficient programming language. Furthermore this author's programming skill is rather limited, and very likely there is room for considerable improvement in terms of the programming. 

The author plans to further improve his Maple program with the help of Prof. I. Anderson. Then we will make the program and the explicit formulas obtained available to the public online. As an interesting application, the author plans to apply these concrete formulas to obtain the first integrals of the full Toda flow on the $E$'s as studied in \cite{GS}.

\bigskip

\begin{bibdiv}
\begin{biblist}

\bib{W-sym}{article}{
   author={Balog, J.},
   author={Feh{\'e}r, L.},
   author={O'Raifeartaigh, L.},
   author={Forg{\'a}cs, P.},
   author={Wipf, A.},
   title={Toda theory and $\scr W$-algebra from a gauged WZNW point of view},
   journal={Ann. Physics},
   volume={203},
   date={1990},
   number={1},
   pages={76--136},
   issn={0003-4916},
}


\bib{ChB}{article}{
   author={Chevalley, C.},
   title={The Betti numbers of the exceptional simple Lie groups},
   conference={
      title={Proceedings of the International Congress of Mathematicians,
      Cambridge, Mass., 1950, vol. 2},
   },
   book={
      publisher={Amer. Math. Soc.},
      place={Providence, R. I.},
   },
   date={1952},
   pages={21--24},
}

\bib{8E8}{article}{
   author={Cederwall, Martin},
   author={Palmkvist, Jakob},
   title={The octic $E_8$ invariant},
   journal={J. Math. Phys.},
   volume={48},
   date={2007},
   number={7},
   pages={073505, 7pp},
   issn={0022-2488},
}


\bib{C}{article}{
   author={Chevalley, Claude},
   title={Invariants of finite groups generated by reflections},
   journal={Amer. J. Math.},
   volume={77},
   date={1955},
   pages={778--782},
   issn={0002-9327},
}


\bib{GS}{article}{
   author={Gekhtman, M. I.},
   author={Shapiro, M. Z.},
   title={Noncommutative and commutative integrability of generic Toda flows
   in simple Lie algebras},
   journal={Comm. Pure Appl. Math.},
   volume={52},
   date={1999},
   number={1},
   pages={53--84},
   issn={0010-3640},
}

\bib{K1}{article}{
   author={Kostant, Bertram},
   title={The principal three-dimensional subgroup and the Betti numbers of
   a complex simple Lie group},
   journal={Amer. J. Math.},
   volume={81},
   date={1959},
   pages={973--1032},
   issn={0002-9327},
}

\bib{K2}{article}{
   author={Kostant, Bertram},
   title={Lie group representations on polynomial rings},
   journal={Amer. J. Math.},
   volume={85},
   date={1963},
   pages={327--404},
   issn={0002-9327},
}

\bib{K3}{article}{
   author={Kostant, Bertram},
   title={On Whittaker vectors and representation theory},
   journal={Invent. Math.},
   volume={48},
   date={1978},
   number={2},
   pages={101--184},
   issn={0020-9910},
}


\bib{LN}{article}{
author={Li, Luen-Chau},
  author = {Nie, Zhaohu},
  title = {Liouville integrability of a class of integrable spin Calogero-Moser systems and exponents of simple Lie algebras}, 
  journal={
  Communications in Mathematical Physics},
   volume={308},
   date={2011},
   number={2},
   pages={415-438},
}

\bib{vavilov}{article}{
   author={Vavilov, N. A.},
   title={Do it yourself structure constants for Lie algebras of types $E_l$},
   language={English translation},
      journal={J. Math. Sci. (N. Y.)},
      volume={120},
      date={2004},
      number={4},
      pages={1513--1548},
      issn={1072-3374},
}


\end{biblist}
\end{bibdiv}
\end{document}